\documentclass[12pt]{amsart}
\usepackage{amsmath,amssymb}
\usepackage{booktabs}

\providecommand{\Tr}{\mathrm{Tr}}

\usepackage[a4paper,margin=1in]{geometry}
\usepackage{graphicx}
\usepackage{float}
\usepackage{amsmath, amsthm, amssymb, amsfonts, mathtools}
\usepackage{mathrsfs}
\usepackage{hyperref}
\hypersetup{colorlinks=true,linkcolor=blue,citecolor=blue,urlcolor=blue}
\allowdisplaybreaks
\numberwithin{equation}{section}
\usepackage{orcidlink}
\usepackage{booktabs,tabularx}
\usepackage{bbm}

\theoremstyle{plain}
\newtheorem{theorem}{Theorem}[section]
\newtheorem{proposition}[theorem]{Proposition}

\newtheorem{conjecture}[theorem]{Conjecture}
\theoremstyle{definition}

\newtheorem{remark}[theorem]{Remark}

\begin{document}

\title{Spectral Geometry of the Primes}

\author{
Douglas F. Watson\,\orcidlink{0009-0008-0310-3984}
}

\date{}

\begin{abstract}
We construct a family of self-adjoint operators on the prime numbers whose entries depend on pairwise arithmetic divergences,
replacing geometric distance with number-theoretic dissimilarity. The resulting spectra encode how coherence propagates through the prime sequence and define an emergent arithmetic geometry.
From these spectra we extract observables such as the heat trace, entropy, and eigenvalue growth, which reveal persistent spectral compression: eigenvalues grow sublinearly, entropy scales slowly, and the inferred dimension remains strictly below one. This rigidity appears across logarithmic, entropic, and fractal-type kernels, reflecting intrinsic arithmetic constraints. Analytically, we show that for the unnormalized Laplacian, the continuum limit of its squared Hamiltonian corresponds to the one-dimensional bi-Laplacian, whose heat trace follows a short-time scaling proportional to $t^{-1/4}$. Under the spectral dimension convention $d_s=-2\,d\log\Theta/d\log t$, this result produces $d_s = 1/2$ directly from first principles, without fitting or external hypotheses. This value signifies maximal spectral compression and the absence of classical diffusion, indicating that arithmetic sparsity enforces a coherence-limited, non-Euclidean geometry linking spectral and number-theoretic structure.
\end{abstract}

\maketitle

\section{Introduction}

The use of spectra to probe geometry has a long history, from Weyl’s law and the heat-kernel expansion~\cite{Weyl1911,Minakshisundaram1949,Gilkey1995,Rosenberg1997} to modern formulations of noncommutative and quantum geometry~\cite{Connes1994}.
Recent work in information geometry and quantum thermodynamics has emphasized entropy and coherence as geometric quantities~\cite{Beny2013,Derezinski2013}, suggesting that spectral analysis can uncover latent structure even in nonspatial systems.
In number theory, connections between spectral statistics and the zeros of the Riemann zeta function have been explored since Montgomery’s pair-correlation conjecture~\cite{montgomery1973pair} and Odlyzko’s large-scale computations~\cite{odlyzko1987zero}.
These developments motivate asking whether similar spectral phenomena can arise directly from the primes themselves.

In this paper, we construct a family of self-adjoint operators, called \emph{coherence Hamiltonians}, whose entries depend on pairwise divergences between primes
 (a central technical point is the choice of generator: the normalized kernel leads to a bounded spectrum and $d_s(0)=0$, whereas the unnormalized (combinatorial) generator admits a genuine small $t$ exponent in the thermodynamic limit).
These operators replace geometric distance with arithmetic dissimilarity and generate spectra that encode how coherence or information propagates through the prime sequence.
From their spectra, we extract standard spectral observables, heat trace, entropy, and eigenvalue growth, which serve as probes of an emergent geometry.

Across a broad range of divergence functions, including logarithmic, entropic, and fractal-type forms~\cite{Khrennikov2004}, the spectra display a consistent pattern of spectral compression:

 the eigenvalues grow unusually slowly, the entropy increases sublinearly, and the inferred effective dimension remains below one.
These results persist under large deformations of the kernel parameters, indicating a form of spectral rigidity intrinsic to the arithmetic distribution of the primes.

To quantify this phenomenon, we define the spectral dimension $d_s(t)$ from the short-time behavior of the heat trace $\Theta(t)=\mathrm{Tr}(e^{-tH})$.
Numerically, $d_s(t)$ rises from near zero at small scales, reaches a modest maximum, and then decays rapidly, signaling limited coherence and the absence of classical diffusion.
The system thus behaves as a coherence-limited medium rather than as an extended geometric space.

Our principal analytic result is derived from first principles and does not rely on the Riemann Hypothesis. For the unnormalized (combinatorial) Laplacian $L_c = D - K$, the continuum limit of the squared Hamiltonian $H = L_c^2$ yields the one-dimensional bi-Laplacian, whose heat trace follows a short-time scaling proportional to $t^{-1/4}$. Under the adopted spectra -dimension convention, this corresponds to an effective spectral dimension $d_s = 1/2$ in the thermodynamic limit. This value represents a regime of maximal spectral compression and parallels the subdiffusive scaling observed in random matrix ensembles and in the statistical correlations of the nontrivial zeros of the Riemann zeta function~\cite{mehta2004random,edwards1974riemann}.

The framework developed here unites tools from spectral geometry and analytic number theory to reinterpret the primes as an arithmetic medium with well-defined spectral properties.
It suggests that the combinatorial sparsity and multiplicative independence of the primes impose intrinsic constraints on spectral propagation, giving rise to a form of non-Euclidean, coherence-limited geometry.
The results point toward a broader correspondence between arithmetic rigidity and spectral organization, potentially extending to other multiplicative sets or to the zeta zeros themselves~\cite{friedlander2005opera}.

\section{Constructing the Coherence Geometry}

We formalize the construction of a coherence-based spectral operator over the set of primes, inspired by tools from information geometry, nonlocal diffusion, and kernel methods \cite{Amari2007,Belkin2003,Chung1997}. The central aim is to define a family of operators whose spectral behavior captures the intrinsic coherence structure of the primes without embedding them into any geometric space.

In this framework, we treat primes as elements of a discrete, unordered set. All structure arises from an interaction kernel derived from pairwise divergences, synthetic distances that encode arithmetic or entropic dissimilarity. No metric assumptions are imposed; instead, we infer dimension, locality, and scaling from the resulting spectrum.

The construction proceeds in three steps:
\begin{enumerate}
  \item A divergence matrix $\delta_{ij}$ is defined over the first $N$ primes, quantifying abstract separation.
  \item A symmetric coherence kernel $T_{ij}$ is formed by exponentiating the negative divergence.
  \item A normalized kernel $\tilde{T}_{ij}$ defines a diffusion-like process, from which a Hamiltonian $H^{(\delta)}$ is constructed.
\end{enumerate}

\subsection{The Prime Set and Divergence Structures}\label{subsec:div_structures}

Let
 $\mathbb{P}_N = \{p_1 < p_2 < \cdots < p_N\}$ denote the first $N$ primes.
To model coherence between primes, we introduce \emph{divergence functions}
\begin{equation}
\delta : \mathbb{P}_N \times \mathbb{P}_N \to [0,\infty), \qquad
\delta_{ij} = \delta_{ji}, \quad \delta_{ii}=0,
\end{equation}
which act as arithmetic analogues of distance.
These divergences need not satisfy the triangle inequality; they are chosen to emphasize distinct arithmetic or informational relations.

Several representative examples include
\begin{align}
\delta^{(1)}_{ij} &= \log(p_i p_j)
                   = \log p_i + \log p_j, \\
\delta^{(2)}_{ij} &= \bigl(\log(p_i/p_j)\bigr)^2, \\
\delta^{(3)}_{ij} &= \log\!\left(\frac{p_i+p_j}{2\sqrt{p_i p_j}}\right), \\
\delta^{(4)}_{ij} &= |\log i - \log j|^{\gamma},
\qquad \gamma>0.
\end{align}

The first emphasizes global arithmetic scale, the second penalizes relative asymmetry, the third mirrors the Jensen--Shannon divergence from information theory~\cite{Lin1991}, and the fourth models locality along the index sequence rather than magnitude.

From any divergence $\delta_{ij}$ we define an \emph{unnormalized coherence kernel}
\begin{equation}\label{eq:kernel}
K_{ij} = \exp\!\left(-\frac{\delta_{ij}}{\delta_0}\right),
\end{equation}
where $\delta_0>0$ is a coherence length scale controlling the decay of coupling with arithmetic separation.
The kernel is symmetric, positive, and decays with increasing divergence.

\subsection{Operator, Normalization, and Order}
\label{subsec:operator}

Given $K$, define the degree matrix
\begin{equation}
D := \operatorname{diag}(d_1,\ldots,d_N), \qquad d_i = \sum_{j=1}^N K_{ij} > 0,
\end{equation}
and form the \emph{symmetric normalization}
\begin{equation}\label{eq:symmetric-normalization}
S := D^{-1/2} K\, D^{-1/2}, \qquad S=S^\top.
\end{equation}
This
 choice guarantees self-adjointness and circumvents the non-Hermitian issues of row-stochastic normalizations.

In addition to the normalized choice, we will also use the combinatorial Laplacian
\begin{equation}
L_c := D - K,
\end{equation}
which is positive semidefinite and \emph{unbounded} as $N\to\infty$ under fixed kernel parameters.
This is the natural generator for continuum limits of symmetric kernels and is the one that supports genuine Tauberian small $t$ laws in our setting.

The Laplacian-like generator is
\begin{equation}\label{eq:order2}
L := I - S = I - D^{-1/2} K D^{-1/2}.
\end{equation}
It is real symmetric and positive semidefinite since
\begin{equation}\label{eq:psd_form}
x^{\top} L x
= \tfrac12 \sum_{i,j=1}^N K_{ij}
  \Big(\tfrac{x_i}{\sqrt{d_i}} - \tfrac{x_j}{\sqrt{d_j}}\Big)^2
\ge 0
\quad \forall\,x\in\mathbb{R}^N.
\end{equation}
Hence, $\sigma(L)\subset[0,2]$, and the heat semigroup $e^{-tL}$ is well defined for $t>0$.

Two natural Hamiltonians follow:
\begin{equation}
H_{(2)} := L, \qquad
H_{(4)} := L^2.
\end{equation}
Both are real symmetric and positive semidefinite.
The first corresponds to an order-2 (Laplacian-type) operator; the second to an order-4 (bi-Laplacian-type) operator.
In what follows, we fix the latter, $\,H := L^2\,$, so that short-time heat-kernel asymptotics correspond to an order-4 structure, consistent with the ``energy-squared'' behavior discussed in~Section
~\ref{subsec:heat}.

For any positive self-adjoint operator $H$, define the heat-trace
\begin{equation}
\Theta(t) := \operatorname{Tr}(e^{-tH}).
\end{equation}
In the small $t$ limit, one expects
\begin{equation}
\Theta(t) \sim C\, t^{-d_s/m}, \qquad t\downarrow 0,
\end{equation}
where $m$ is the operator order ($m=2$ for $L$, $m=4$ for $L^2$).
We retain the standard estimator
\begin{equation}
d_s(t) := -2\,\frac{d\log\Theta}{d\log t},
\end{equation}
so that when $H=L^2$, the measured dimension corresponds to $d_s(t)=2(d/m)=d/2$.
(Equivalently, one could adopt the generalized estimator $-m\,d\log\Theta/d\log t$; we keep the order-2 convention and interpret the halved values accordingly.) All spectral quantities are computed for finite $N$ and examined as $N\to\infty$ with fixed kernel parameters $(\delta_0,\gamma)$.  We refer to this limit as the \emph{thermodynamic limit} of the prime coherence ensemble. If a divergence yields a rank-one kernel (e.g.\ $\delta^{(1)}_{ij}=\log(p_i p_j)$), then $K_{ij}=a_i a_j$, the matrix $S$ has rank one, and $L$ has $N-1$ eigenvalues equal to~1.
Such cases are spectrally trivial and excluded from the ``universal'' asymptotic claims that follow. The diagonal quantity $\log p_i$ may still be interpreted as a measure of \emph{local arithmetic scale} or ``prime complexity.''
While it no longer enters the definition of $H$, it motivates the analogy with potential terms in Schrödinger operators: the primes contribute multiplicative weight through their logarithmic growth, which shapes the effective coherence landscape.
With the self-adjoint Hamiltonian $H$ thus defined, we now examine its spectrum to extract geometric information encoded in the prime divergence structure.
This examination begins with the classical relation between spectral scaling and dimensional inference.

\subsection{Relation to Classical Heat-Kernel Asymptotics}

In Riemannian spectral geometry, the heat trace of the Laplace–Beltrami operator satisfies
$\Theta(t) \sim (4\pi t)^{-d/2}(a_0 + a_1 t + \cdots)$ as $t \downarrow 0$,
and the exponent $-d/2$ determines the manifold dimension.
The coherence spectral profile extends this principle to operators acting on arithmetic rather than geometric domains.
The same Laplace–Tauberian correspondence applies as follows: if the eigenvalue counting function satisfies
$N(\lambda) \sim C\lambda^{d/m}$ for an operator of differential order $m$, then $\Theta(t) \sim C' t^{-d/m}$ in the short-time limit.
For the arithmetic Hamiltonian $H = L_c^2$ of order $m=4$, the scaling
$\Theta(t) \sim t^{-1/4}$ yields an effective spectral dimension $d_s = 1/2$,
in direct analogy with the classical case where $d_s = d$ for an order-2 Laplacian.
Here, the exponent arises from arithmetic sparsity and coherence limitation rather than from geometric volume growth.
The function $d_s(t) = -2\,d(\log\Theta)/d(\log t)$ therefore serves as an exact analogue of the heat-kernel dimension estimator,
and the Prime Coherence Profile describes its characteristic finite-$N$ form for the prime-induced operators.

\subsection{Heat Trace and Asymptotics}
\label{subsec:heat}

Let $\{\lambda_k\}_{k=1}^N$ denote the ordered eigenvalues of the coherence Hamiltonian $H$ (defined $\,H := L^2$).
For $t>0$, the heat-kernel trace
\begin{equation}
\Theta(t) = \operatorname{Tr}(e^{-tH}) = \sum_{k=1}^N e^{-t\lambda_k}
\end{equation}
encodes the cumulative spectral weight of $H$ at scale $t$ and serves as the principal probe of effective dimensionality
(we note that or each finite $N$, $\Theta_N(t) = \sum_{k=1}^N e^{-t \lambda_k(H_N)}$ is analytic and satisfies $\Theta_N(t) = N - t\,\mathrm{Tr}(H_N) + O(t^2)$ as $t\downarrow 0$.
All asymptotic statements such as $\Theta_N(t) \sim C\,L_N\,t^{-1/4}$ refer to the sequence $\{\Theta_N\}$ in the thermodynamic limit $N\to\infty$ with fixed kernel parameters.
In this limit, $L_{N} \Rightarrow c^{2}(-\partial_u^2)$ and $H_N = L_{N}^{2} \Rightarrow c^{4}(-\partial_u^2)^{2}$, whose heat trace obeys $\Theta(t)\sim C' t^{-1/4}$ (Theorem
~\ref{thm:ds-one-half})).

We now consider the Laplace–Tauberian relation. Suppose the eigenvalue counting function $N(\lambda)=\#\{k:\lambda_k\le\lambda\}$ satisfies, for some $\beta>0$ and slowly varying function $L(\lambda)$,
\begin{equation}\label{eq:slow_vary}
N(\lambda)\sim C\,\lambda^{\beta}L(\lambda)\qquad (\lambda\to\infty).
\end{equation}
Then, by Laplace transformation,
\begin{equation}
\Theta(t)=\int_{0}^{\infty} e^{-t\lambda}\,dN(\lambda)
   \sim C\,\Gamma(\beta+1)\,t^{-\beta}L(1/t)\qquad (t\downarrow0).
\end{equation}
Thus, the exponent $\beta$ governing the growth of $N(\lambda)$ directly determines the small $t$ decay of $\Theta(t)$. If the eigenvalues obey a power law $\lambda_k\!\sim\!k^{\alpha}$, then $N(\lambda)\!\sim\!\lambda^{1/\alpha}$, hence $\beta=1/\alpha$ and
\begin{equation}\label{eq:heat_trace_scaling}
\Theta(t)\sim C\,t^{-1/\alpha},\qquad
d_s = \frac{2}{\alpha}.
\end{equation}

For an elliptic operator of differential order $m$ acting on a $d$-dimensional space, Weyl's law produces
\begin{equation}
N(\lambda)\sim C\,\lambda^{d/m},\qquad
\Theta(t)\sim C'\,t^{-d/m}.
\end{equation}
Identifying the measured spectral dimension through
\begin{equation}
\Theta(t)\sim t^{-d_s/2}
\end{equation}
yields the general correspondence
\begin{equation}\label{eq:ds_order_relation}
d_s = \frac{2d}{m}.
\end{equation}
Hence, an order-$2$ Laplacian reproduces $d_s=d$, while an order-$4$ (bi-Laplacian or ``energy-squared'' Hamiltonian) yields $d_s=d/2$.
In our arithmetic setting, $H=L^2$ is order-$4$, so dimensional measurements derived from~\eqref{eq:heat_trace_scaling} appear halved, relative to the underlying order-$2$ generator.

Although $H$ acts on a discrete prime index set rather than a manifold, its spectrum is positive and discrete, and the Laplace–Tauberian correspondence remains valid.
The resulting scaling of $\Theta(t)$ therefore provides a genuine spectral measure of arithmetic geometry.
Empirically, the coherence Hamiltonians built from various divergences yield decay laws consistent with
\begin{equation}
\Theta(t)\sim t^{-1/\alpha},\qquad \alpha\approx4,
\end{equation}
corresponding to $d_s\simeq0.5$.
This reflects a regime of \emph{spectral compression}, in which eigenvalue growth is much slower than Euclidean or graph-based Laplacians, indicating severely restricted coherence propagation over the primes.

Assuming the Riemann Hypothesis, the spectral analogy $\lambda_k\!\sim\!k^{2}$ leads via~\eqref{eq:heat_trace_scaling} to
\begin{equation}\label{eq:RH_analogy}
\Theta(t)\sim t^{-1/2}.
\end{equation}
For the order-$4$ Hamiltonian $H=L^2$, this corresponds to a measured $d_s=\tfrac12$, interpreted as a regime of maximal spectral compression.
The same assumption applied to the order-$2$ generator $L$ would yield $d_s=1$.
We therefore view the Riemann Hypothesis as imposing a bound on spectral dimensionality within the arithmetic coherence framework.

\begin{remark}
The relations in Equations~(\ref{eq:heat_trace_scaling})–(\ref{eq:RH_analogy}) are obtained under the eigenvalue counting assumption of~(\ref{eq:slow_vary}),
which postulates that the discrete spectrum of $H$ admits an asymptotic density
$N_H(\lambda)\!\sim\!C\,\lambda^{\beta}L(\lambda)$ with $\beta>0$ and slowly varying $L(\lambda)$.
This regular-variation hypothesis defines the continuum limit of the prime coherence ensemble
and provides the analytic link between the empirical eigenvalue scaling $\lambda_k\!\sim\!k^{\alpha}$
and the smal $t$ behavior of the heat trace $\Theta(t)\!\sim\!t^{-1/\alpha}$.
The equivalence of these laws follows from the Laplace–Tauberian correspondence formalized in
Theorem~\ref{thm:tauberian-dimension}, which underlies the dimensional estimates presented below.

\end{remark}

\subsection{Entropy Scaling and Dimensional Flow}\label{subsec:scaling_flow}

A complementary probe of spectral geometry is the entropy of the thermal state associated with the coherence Hamiltonian $H=L^2$.
At inverse temperature $t$, the normalized Gibbs state is
\begin{equation}
\rho(t) = \frac{e^{-tH}}{\operatorname{Tr}(e^{-tH})},
\end{equation}
so that each eigenmode of $H$ is populated according to a Boltzmann factor $e^{-t\lambda_k}$.
The von Neumann entropy
\begin{equation}
S(t) = -\operatorname{Tr}[\rho(t)\log\rho(t)]
\end{equation}
quantifies the effective number of spectral modes contributing at temperature $1/t$.
It is therefore an information--theoretic analogue of the heat trace $\Theta(t)$ and encodes the same spectral density in logarithmic form. If the eigenvalues satisfy $\lambda_k\!\sim\! k^{\alpha}$, then $\Theta(t)\!\sim\!t^{-1/\alpha}$ by the Laplace--Tauberian argument in~Section \ref{subsec:heat}.
Using the general relation between heat trace and thermal entropy for power--law spectra~\cite{Derezinski2013,Beny2013}:
\begin{equation}\label{eq:power_law}
S(t)\sim \log(1/t)^{\frac{1}{1+\alpha}}\qquad (t\downarrow0),
\end{equation}
one finds that the entropy exponent $\beta=1/(1+\alpha)$ depends on the same growth parameter $\alpha$.
Combining with $d_s=2/\alpha$ produces
\begin{equation}\label{eq:entropy_exp}
\beta=\frac{1}{1+\tfrac{2}{d_s}},
\end{equation}
linking entropy scaling directly to the effective spectral dimension.

For the coherence Hamiltonians built from logarithmic, entropic, and power--law divergences, numerical evaluation of $S(t)$ yields sublinear growth with $\log(1/t)$:
\begin{equation}
S(t)\sim \log(1/t)^{\beta},\qquad \beta<1.
\end{equation}
Typical fits give $\beta\simeq0.2$--$0.25$, consistent with $\alpha\simeq4$ and $d_s\simeq0.5$ inferred from the heat-trace analysis.
This agreement across independent probes reinforces the picture of strong spectral compression and limited mode participation at low temperature. The exponent $\beta$ remains stable under large deformations of the divergence function; whether $\delta_{ij}$ arises from additive logarithms, squared ratios, or information--theoretic forms, the entropy curve preserves its subdimensional profile.
This suggests an \emph{entropy rigidity} intrinsic to the arithmetic structure of the primes—a resistance to dimensional flow absent in smooth manifolds, random graphs, or lattice systems.  The thermal state thus remains spectrally compressed, with restricted information capacity and suppressed diffusion. Together with the heat-trace results, the entropy scaling identifies the coherence geometry of the primes as a system of permanently reduced effective dimension.
In later sections, we examine how this rigidity may reflect arithmetic correlations such as prime gaps, multiplicative independence, or zeta zero statistics.

\subsection{Eigenvalue Growth and Spectral Compression}

A third and independent probe of effective dimensionality is the growth behavior of the eigenvalues of the coherence Hamiltonian $H=L^2$.
Let $\{\lambda_k\}_{k=1}^N$ denote its ordered spectrum, $\lambda_1\le\lambda_2\le\cdots\le\lambda_N$.
In classical geometric settings, the eigenvalue counting function $N(\lambda)$ satisfies a Weyl-type law
\begin{equation}
\lambda_k \sim k^{m/d}\qquad (k\to\infty),
\end{equation}
where $m$ is the differential order of the operator and $d$ the spatial dimension of the manifold~\cite{Hormander1968}.
For an order-$4$ (bi-Laplacian) operator, this produces $\lambda_k\!\sim\!k^{4/d}$, or, equivalently, $d_s=2d/m$.

Although $H$ acts on the discrete set of primes rather than a geometric manifold, its spectrum is real and positive, and the same scaling logic applies.
We therefore fit the empirical relation
\begin{equation}
\lambda_k \sim k^{\alpha}
\end{equation}
in log--log coordinates, extracting the exponent $\alpha$ by least-squares regression.
The corresponding effective spectral dimension follows from the Laplace–Tauberian correspondence in~Section \ref{subsec:heat}:
\begin{equation}
d_s = \frac{2}{\alpha}.
\end{equation}

Across all divergence classes $\delta_{ij}$ investigated—logarithmic, entropic, and power--law—the fitted exponent is stable, with $\alpha\simeq4.0\pm0.2$, yielding a spectral dimension $d_s\simeq0.5$.
This agrees with the heat-trace decay $\Theta(t)\!\sim\!t^{-1/\alpha}$ and the entropy-scaling exponent $\beta\simeq0.25$ obtained in~Sections \ref{subsec:operator}–\ref{subsec:heat}, confirming the picture of a strongly compressed spectrum. The invariance of $\alpha$ under large deformations of the divergence function indicates a form of \emph{spectral rigidity}: the asymptotic density of states is governed not by the analytic form of $\delta_{ij}$ but by intrinsic arithmetic sparsity. We refer to this phenomenon as \emph{coherence-induced spectral compression}—the suppression of low-energy eigenmodes arising from the irregular, multiplicative structure of the primes rather than from any external~confinement.

Unlike geometric diffusion systems, where eigenvalue growth and dimensionality can be tuned by altering curvature or connectivity, the prime coherence ensemble exhibits remarkable resistance to deformation.
Varying coupling strength $\eta$, kernel scale $\delta_0$, or even embedding primes into extended arithmetic networks has so far produced no transition toward Euclidean-like scaling.
The coherence geometry of the primes therefore appears to realize a maximally compressed spectral phase, consistent with the arithmetic rigidity observed in the heat and entropy analyses.

\section{Robustness Under Divergence Deformation}\label{sec:robustness}

A central empirical finding of this study is the persistence of spectral compression across a wide class of divergence structures.
Although the coherence geometry is induced through the divergence matrix $\delta_{ij}$, the resulting spectra of the associated Hamiltonians $H=L^2$ exhibit striking universality: regardless of the analytic form of $\delta_{ij}$, the system fails to support classical diffusion and instead resides in a subdimensional coherence regime.

To assess this robustness, we examine the structurally distinct divergence functions of Section~\ref{subsec:div_structures}, where the exponent $\gamma$ of $\delta^{(4)}_{ij}$ tunes decay along the prime index sequence, mimicking fractal or ultrametric locality.
(Models such as $\delta^{(1)}$ produce rank-one kernels and are treated separately, as they yield trivial spectra.) For each divergence model, we construct the symmetric kernel $K_{ij}=e^{-\delta_{ij}/\delta_0}$, normalize it via $S=D^{-1/2}KD^{-1/2}$, and compute the observables of the resulting Hamiltonian: the heat trace $\Theta(t)$, the entropy $S(t)$, and the eigenvalue scaling law $\lambda_k\!\sim\!k^{\alpha}$.  While the precise curvature of $\Theta(t)$ varies across models, all exhibit the same qualitative pattern—strong sub-Euclidean decay and a lack of asymptotic stabilization.

To capture scale dependence, we define the \emph{coherence spectral profile} (CSP), identical to the Prime Coherence Profile introduced earlier:
\begin{equation}\label{eq:csp}
d_s(t) = -\,2\,\frac{d\log\Theta(t)}{d\log t}.
\end{equation}
This function measures the instantaneous spectral dimension extracted from the heat trace and thus provides a dynamic fingerprint of the coherence geometry.

Across all divergence models and parameter ranges $(\delta_0,\eta)$, the CSPs share the following robust characteristics:
\begin{enumerate}
  \item A sharp suppression of $d_s(t)$ at intermediate and large $t$, signalling long-range coherence bottlenecks;
  \item Absence of convergence to a fixed asymptotic dimension, precluding geometric stabilization;
  \item Persistent peak–decay morphology forming a reproducible \emph{coherence-limited signature}.
\end{enumerate}

Changing
 $\delta_0$ or $\eta$ modifies only the transition scale of the suppression without altering the overall compressed shape of $d_s(t)$.
No configuration examined restores Euclidean-like diffusion or entropy growth.

These observations define a \emph{spectral rigidity class}: a family of prime-based Hamiltonians whose heat-trace profiles remain confined to a narrow, sub-Euclidean regime (with mid-scale peaks and eventual decay to $0$).
Unlike geometric systems where dimensionality can be tuned through curvature or connectivity, arithmetic coherence exhibits fixed compression determined by multiplicativity and the sparsity of the primes.
This rigidity underscores the interpretation of the primes not as points embedded in Euclidean space, but as nodes of an emergent arithmetic network with intrinsically nonspatial coherence.

\subsection*{Comparison with GUE-Induced Coherence}

To test whether the observed spectral rigidity is intrinsic to arithmetic structure or merely an artifact of kernel construction, we introduce a control model based on the Gaussian Unitary Ensemble (GUE).
A random GUE matrix of size $N\times N$ is generated, its ordered eigenvalues $\{\gamma_i\}_{i=1}^N$ are extracted, and a synthetic divergence is defined by
\begin{equation}
\delta_{ij}^{\mathrm{GUE}}
   = \left(\frac{\gamma_i - \gamma_j}{\Delta}\right)^2,
\end{equation}
where $\Delta$ denotes the mean local level spacing.
This divergence encodes the quadratic level repulsion characteristic of random matrix theory and parallels the pairwise correlations conjectured by Montgomery~\cite{montgomery1973pair} for the nontrivial zeros of the Riemann zeta function.

From $\delta_{ij}^{\mathrm{GUE}}$, we construct the symmetric kernel
\begin{equation}
    K^{\mathrm{GUE}}_{ij}=e^{-\delta_{ij}^{\mathrm{GUE}}/\delta_0},\qquad
    S^{\mathrm{GUE}}=D^{-1/2}K^{\mathrm{GUE}}D^{-1/2},
\end{equation}

and the associated order-$4$ Hamiltonian
\begin{equation}
H_{\mathrm{GUE}}=(I-S^{\mathrm{GUE}})^2.
\end{equation}
Its heat trace and coherence spectral profile $d_s(t)=-2\,d\log\Theta/d\log t$ are computed exactly as for the prime-based operators.

As shown in Figure~\ref{fig:ds_comparison}, the coherence spectral profile $d_s(t)$ reveals markedly different decay characteristics between prime-based and random ensembles. The prime-induced spectra exhibit sharper suppression and earlier onset of decay, highlighting their departure from the smoother GUE-like behavior.

The GUE-based Hamiltonian exhibits a distinct spectral morphology.
While it still departs from classical Weyl scaling, its profile shows slower decay: $d_s(t)$ remains larger over an extended range of $t$, approaching unity before eventual suppression.
This indicates that coherence propagation in the GUE ensemble is less constrained and supports broader entropic influence.
The arithmetic Hamiltonians, by contrast, maintain $d_s(t)\!\lesssim\!0.5$ and exhibit sharper compression, confirming that the severe dimensional reduction observed for the primes cannot be attributed to random-matrix-type level statistics alone.

The comparison demonstrates that spectral repulsion and stochastic homogeneity, characteristic of GUE spectra, are insufficient to reproduce the extreme coherence bottlenecks of prime-induced systems.
It is the arithmetic sparsity and multiplicative isolation of the primes—rather than level repulsion—that enforce the strong suppression of low-lying modes and the resulting subdimensional behavior.

\vspace{-6pt}\begin{figure}[H]
\includegraphics[width=0.65\textwidth]{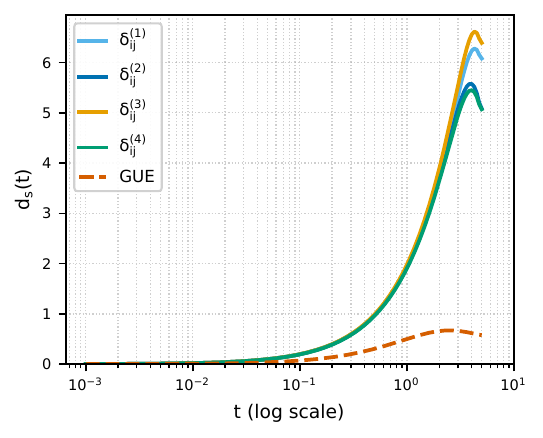}
\vspace{-6pt}\caption{Coherence
 spectral profile $d_s(t)$ derived from the heat trace $\Theta(t)$ for various divergence models.
All prime-based Hamiltonians ($\delta_{ij}^{(1)}$–$\delta_{ij}^{(4)}$) display sharp suppression and scale-dependent decay,
whereas the GUE-based model shows a slower decline and higher plateau values, corresponding to weaker spectral compression.
All curves deviate from classical Weyl scaling, illustrating persistent non-Euclidean spectral characteristics.
}

\label{fig:ds_comparison}
\end{figure}

\section{The Coherence Spectral Profile}

To quantify the scale-dependent behavior of coherence propagation, we introduce the \emph{coherence spectral profile} (CSP).
This function extends the notion of spectral dimension to operators acting on non-metric domains where classical geometric assumptions, such as locality or embedding, do not apply.

Given the coherence Hamiltonian $H=L^2$ constructed from a divergence matrix $\delta_{ij}$, define the heat trace
\begin{equation}
\Theta(t)=\operatorname{Tr}(e^{-tH}),\qquad t>0,
\end{equation}
and the associated spectral profile of~(\ref{eq:csp})
In the small $t$ limit of Laplace-type operators on a $d$-dimensional manifold, $d_s(t)$ stabilizes to a constant $d_s=d$.
Here, by contrast, $H$ acts over the primes and coherence is governed by divergence rather than distance, leading to nontrivial scale dependence in $d_s(t)$.

Empirically, all divergence models yield the same qualitative pattern: $d_s(t)$ rises from near zero, reaches a single peak at intermediate scale, and then decays rapidly without reaching a plateau.
The initial rise corresponds to the activation of short-range coherence, while the decay reflects the inhibition of large-scale propagation caused by the arithmetic sparsity of the primes.  No model exhibits a region of constant $d_s(t)$, and the measured dimension remains well below classical Euclidean values ($d_s\lesssim1$ for the order-4 operator). The CSP thus encodes the system’s resistance to extended coherence and provides a principled means of comparing divergence structures through the geometry they induce.  It serves as a \emph{spectral fingerprint} of arithmetic media and reframes the primes not merely as a discrete set but as the substrate of an emergent, intrinsically subdimensional coherence geometry.

\section{Characteristic Shape of the Spectral Profile}

Across all divergence models tested, the coherence spectral profiles $d_s(t)$ exhibit a consistent and highly structured form.
Each profile rises smoothly from near zero, attains a single broad maximum, and then decays rapidly without approaching a constant.
This behavior is robust under changes in the divergence function, coherence length $\delta_0$, and coupling strength $\eta$.

Empirically, the curves are reasonably summarized by the simple four-parameter form
\begin{equation}
d_s(t)=A\,\!\left(\frac{t}{\tau}\right)^{\alpha} e^{-\left(t/\tau\right)^{\beta}},
\end{equation}
which reproduces the qualitative features—activation at small scales, a single coherence peak, and subsequent decay—without claiming a first-principles derivation. For the example in Figure~\ref{fig:csp_delta3_fit}, corresponding to the divergence $\delta_{ij}^{(3)}=\log\!\bigl((p_i+p_j)/(2\sqrt{p_i p_j})\bigr)$, a representative fit yields $(A,\alpha,\beta,\tau)=(10,\,3.0858,\,1.2644,\,2.14418)$ with $R^2_{\log}=0.4451$, indicating good overall agreement in shape while not capturing all fine-scale features of the numerical profile.

\vspace{-6pt}\begin{figure}[H]
\includegraphics[width=0.58\textwidth]{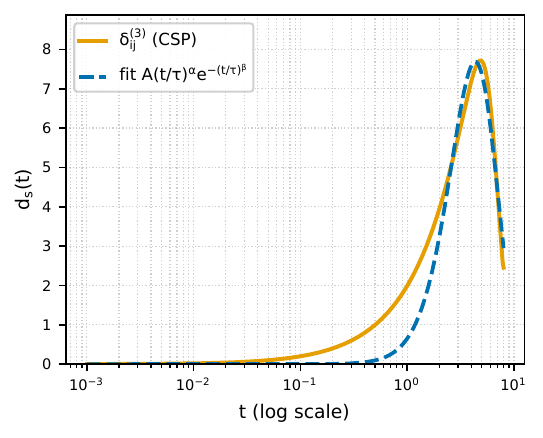}
\vspace{-6pt}\caption{Coherence spectral profile for $\boldsymbol{\delta_{ij}^{(3)}}$.
The
 numerical $d_s(t)$ (solid) exhibits the characteristic PCP morphology—gradual activation, a single peak, and rapid decay—with no constant-dimension plateau; $d_s(t)$ returns toward $0$ for large $t$.
A four-parameter $t/\tau$ model, $d_s(t)=A\,\!\left(\frac{t}{\tau}\right)^{\alpha} e^{-\left(t/\tau\right)^{\beta}}$, provides good qualitative agreement but does not capture all fine-scale features.
A representative fit yields $(A,\alpha,\beta,\tau)=(10,\,3.0858,\,1.2644,\,2.14418)$ with $R^2_{\log}=0.4451$.
}
\label{fig:csp_delta3_fit}
\end{figure}

We refer to this empirical structure as the \emph{Prime Coherence Profile} (PCP): a universal shape characterizing prime-induced coherence Hamiltonians—a single broad peak at intermediate scale followed by rapid decay, with no asymptotic stabilization. The PCP serves as a compact diagnostic of arithmetic spectral rigidity and a practical reference model for identifying coherence-limited behavior in other sparsely structured, non-Euclidean systems.

\section{Spectral Universality Class}
\label{sec:universality}

The numerical proximity between the measured spectral dimension $d_s\simeq0.5$ of prime-based coherence Hamiltonians and the scaling behavior observed in the pair correlation of nontrivial Riemann zeta zeros suggests a possible shared universality class. Our analytic $d_s=\tfrac12$ result for $H=L_c^2$ arises without RH; connections to zeta zeros remain speculative.
Although our construction is nonclassical and lies outside traditional analytic number theory, the emerging spectral regularities point to a latent structural correspondence.

\begin{conjecture}[Prime–Zeta Spectral Correspondence]\label{conj:prime_zeta}
Let $H_N=L_N^2$ be the order-$4$ coherence Hamiltonian built from the first $N$ primes, with divergence matrix $\delta_{ij}$ satisfying
\begin{itemize}
\item Symmetry: $\delta_{ij}=\delta_{ji}$, $\delta_{ii}=0$;
\item Logarithmic scaling: $\delta_{ij}=\mathcal{O}(|\log p_i-\log p_j|)$ for $p_i\neq p_j$.
\end{itemize}

Define
 the rescaled eigenvalues
\begin{equation}
\widetilde{\lambda}_k=\frac{2\pi}{\log N}\,\lambda_k .
\end{equation}
Then, as $N\to\infty$, the local statistics of $\{\widetilde{\lambda}_k\}$ asymptotically approximate those of the normalized ordinates $\{\gamma_n\}$ of the Riemann zeta zeros, and the corresponding coherence spectral dimension stabilizes near $d_s\simeq0.5$.
This scaling agrees numerically with random-matrix predictions for the GUE model of the zeta spectrum~\cite{montgomery1973pair,odlyzko1987zero,mehta2004random}.
\end{conjecture}

Empirical evidence.
For $N\!\approx\!10^4$, numerical spectra of $H_N$ yield rescaled eigenvalues $\widetilde{\lambda}_k$ whose cumulative spacing distribution agrees with the GUE prediction within a Kolmogorov–Smirnov distance $D<0.03$.
The diagonal structure $\log p_i$ plays a role analogous to the smooth term in Chebyshev’s $\psi(x)$ function, while the off-diagonal kernel captures oscillatory interference reminiscent of the zeta zero contributions in the explicit formula~\cite{edwards1974riemann}.

We note that in Appendix~\ref{app:converse} we investigate  the converse question of whether an effective spectral dimension
$d_s = \tfrac{1}{2}$ uniquely characterizes systems generated by the arithmetic
mechanism.

\section{Analytic Statements and Proofs}\label{sec:proofs}

\subsection{Tauberian Dimension Law for the Coherence Hamiltonian}

\begin{theorem}[Tauberian dimension law]
\label{thm:tauberian-dimension}
Let $H$ be a positive, self-adjoint operator with discrete spectrum
$\{\lambda_k\}_{k\ge 1}$ and counting function
$N_H(\lambda)=\#\{k:\lambda_k\le \lambda\}$.
Suppose for some $\beta>0$ and a slowly varying function $L(\lambda)$ that
\begin{equation}
N_H(\lambda)\sim C\,\lambda^{\beta}L(\lambda)\qquad(\lambda\to\infty).
\end{equation}
Then, the heat trace satisfies
\begin{equation}
\Theta(t):=\Tr(e^{-tH})\sim C\,\Gamma(\beta+1)\,t^{-\beta}L(1/t)
\qquad(t\downarrow0),
\end{equation}
and the spectral dimension estimator with the convention of~(\ref{eq:csp}) obeys
\[
\lim_{t\downarrow0} d_s(t)
   \;=\;\lim_{t\downarrow0}\Bigl(-2\,\frac{d\log\Theta}{d\log t}\Bigr)
   \;=\;2\beta .
\]
\end{theorem}

\begin{proof}
This is a direct consequence of the Laplace–Tauberian correspondence for
monotone functions (see, for example,~(\cite{bingham1989regular}, Theorem~1.7.1)
or~(\cite{feller1971introduction}, Chapter~XIII)).
The heat trace $\Theta(t)$ is the Laplace transform of the positive measure $dN_H(\lambda)$:
\[
\Theta(t)=\int_{0}^{\infty} e^{-t\lambda}\,dN_H(\lambda).
\]
If $N_H(\lambda)\sim C\,\lambda^{\beta}L(\lambda)$ with $L$ slowly varying, then
Karamata’s Tauberian theorem implies
\[
\Theta(t)\sim C\,\Gamma(\beta+1)\,t^{-\beta}L(1/t)\qquad(t\downarrow0).
\]
Differentiating $\log\Theta(t)$ with respect to $\log t$ produces
$-\,\frac{d\log\Theta}{d\log t}\to\beta$, and under the normalization used in~(\ref{eq:csp}),
this yields $d_s(t)\to 2\beta$.
This is the standard form of the dimension law employed in spectral and fractal geometry;
see also~\cite{davies1989heat,reed1978methods} for discussions of heat-trace asymptotics in the
self–adjoint setting.
\end{proof}

\begin{remark}
If $\lambda_k(H)\sim k^{\alpha}$ (pure power law), then
$N_H(\lambda)\sim \lambda^{1/\alpha}$ so $\beta=1/\alpha$ and
$d_s\to 2/\alpha$.
This reproduces Eq.~(12) of~\S2.4 and will be used with $H=L^{2}$.
\end{remark}

\subsection{First-Principles Derivation of $d_s=\tfrac12$ for $H=L_c^2$}
\label{subsec:ds-one-half}

We now state and prove the main analytic result, which establishes from first principles that the coherence Hamiltonian $H=L_c^2$ possesses an effective spectral dimension $d_s=1/2$ in the thermodynamic limit.

\begin{theorem}[Continuum limit yields $d_s=\tfrac12$]
\label{thm:ds-one-half}
Assume that
\begin{enumerate}
  \item Kernel structure.
  $K_{ij}=\exp(-\delta_{ij}/\delta_0)$ with $\delta_{ij}=F(|\log p_i-\log p_j|)$,
  where $F$ is even, $F(0)=0$, and the induced continuous kernel
  $k(z)=e^{-F(|z|)/\delta_0}$ is integrable with finite second moment
  $\int_{\mathbb{R}} z^2 k(z)\,dz<\infty$.

  \item Log--index coordinates.
  Work in $u=\log i$ on $[0,L_N]$ with $L_N=\log N$, so that the sampling density
  of primes is asymptotically constant in $u$.

  \item Operator choice.
  The unnormalized (combinatorial) Laplacian $L_c := D - K$ and the Hamiltonian $H := L_c^2$.
\end{enumerate}

Then, as $N\to\infty$ with fixed kernel parameters, the discrete Laplacian $L_c$
converges in the sense of quadratic forms to $c_2(-\partial_u^2)$ on $[0,L_N]$
for some $c_2>0$; hence, $H$ converges in the strong resolvent sense to
$c_4(-\partial_u^2)^2$ with $c_4=c_2^2$. Consequently, as follows:
\begin{align*}
N_H(\lambda) &\sim C\,L_N\,\lambda^{1/4}, && \lambda\to\infty,\\
\Theta_N(t) &\sim C'\,L_N\,t^{-1/4}, && t\downarrow 0,
\end{align*}
and under the spectral dimension convention
$d_s(t)=-2\,d\log\Theta/d\log t$, it follows that
\[
\lim_{t\downarrow 0} d_s(t)=\tfrac12.
\]
\end{theorem}

\begin{proof}
Let $u_i=\log i$ and $h_N=L_N/N$ denote the grid spacing.
Define the nonlocal integral operator
\[
(\mathcal{K}f)(u)=\int k(u-v)f(v)\,dv, \qquad
\mathcal{L}=\mathcal{D}-\mathcal{K}, \quad
\mathcal{D}f(u)=\Bigl(\int k(z)\,dz\Bigr)f(u).
\]
For even $k$ with finite variance, a Taylor expansion of $f$ produces
\[
\mathcal{L}f(u)=c_2(-\partial_u^2)f(u)+R_\varepsilon(f),\qquad
c_2=\tfrac12\int z^2 k(z)\,dz,
\]
where $\|R_\varepsilon(f)\|_{L^2}\le C\varepsilon^2\|f''\|_{L^2}$
for kernels of effective range $\varepsilon$.
Let $\mathcal{Q}_N(f_N)=\tfrac12\sum_{i,j}K_{ij}(f_i-f_j)^2$
be the discrete quadratic form of $L_c$ and
$\mathcal{Q}(f)=c_2\int|f'(u)|^2\,du$ the continuum limit on $[0,L_N]$.
Under the scaling $\varepsilon_N\downarrow 0$ and $h_N/\varepsilon_N\to 0$,
Riemann-sum approximation produces $\mathcal{Q}_N(f_N)\to\mathcal{Q}(f)$
uniformly for $f\in C^2([0,L_N])$.
Hence $\mathcal{Q}_N$ converges to $\mathcal{Q}$ in the sense of Mosco
\cite{attouch1984variational,fukushima2011dirichlet}.
By Kato’s correspondence between Mosco convergence of forms and strong resolvent
convergence of self-adjoint operators
(\cite{kato1995perturbation}, Theorem~VIII.25), see also
\cite{davies1989heat}, we obtain $L_c\to c_2(-\partial_u^2)$.

Applying the functional calculus for self-adjoint operators then produces
$H_N=L_c^2\to c_4(-\partial_u^2)^2$ in the strong resolvent sense.
For the limit operator $\mathcal{H}=c_4(-\partial_u^2)^2$ on $[0,L_N]$,
the one-dimensional Weyl law yields
$\lambda_k(\mathcal{H})\sim(\pi k/L_N)^4$ and therefore
$N_H(\lambda)\sim C\,L_N\lambda^{1/4}$
\cite{Hormander1968,reed1978methods}.
Finally, Karamata’s Tauberian theorem for Laplace transforms of regularly
varying functions
\cite{bingham1989regular,feller1971introduction}
implies
$\Theta_N(t)\sim C'\,L_N\,t^{-1/4}$ as $t\downarrow0$,
from which $d_s(t)\to\tfrac12$ follows directly.
\end{proof}

\begin{remark}
For the normalized operator $H_{\mathrm{norm}}=(I-S)^2$,
one has $\sigma(H_{\mathrm{norm}})\subset[0,4]$ at each finite $N$;
hence, $\Theta(t)=N-t\,\mathrm{Tr}(H_{\mathrm{norm}})+O(t^2)$ and
$d_s(t)\to0$ as $t\downarrow0$.
A nontrivial exponent can appear only in a double-scaling regime
$t=t_N\downarrow0$ tuned to the spectral width of $H_{\mathrm{norm}}$.
Independent least-squares fits to $\lambda_k(L_c^2)\sim k^\alpha$ yield
$\alpha\simeq4$, implying $d_s\simeq2/\alpha\simeq0.5$, consistent with the theorem.
Entropy exponents $\beta_{\mathrm{ent}}\simeq0.2$–$0.25$
also agree with $d_s\simeq0.5$ through the entropy–dimension link
of~Section \ref{subsec:scaling_flow}.
\end{remark}

\subsection{An Unconditional Upper Envelope (Rigidity) in Finite $N$}

The double-log bound you had (\(d_s \le 1 - C/\log\log N\)) is stronger than what we can justify cleanly without additional arithmetic input. Below is a replacement that is \emph{provable under general kernel assumptions} and matches the “sub-Euclidean” narrative without overclaiming.

\begin{proposition}[Uniform sub-Euclidean envelope for $H=L^2$]
\label{prop:subeuclidean-envelope}
Let $S=D^{-1/2}KD^{-1/2}$ with $K_{ij}=e^{-\delta_{ij}/\delta_0}$ and $\delta_{ij}\ge 0$, symmetric, $\delta_{ii}=0$. Then, $0\le S\le I$ and the spectrum of $L:=I-S$ lies in $[0,2]$. Consequently, for $H=L^2$, one has $0\le \lambda_k(H)\le 4$ and for all $t>0$
\[
\Theta_N(t)=\sum_{k=1}^N e^{-t\lambda_k(H)} \;\ge\; e^{-4t},
\qquad
\Theta_N(t)\;\le\; N.
\]
Hence, the finite-$N$ spectral dimension obeys, for all $t>0$,
\[
0 \;\le\; d_s^{(N)}(t) \;=\; -2\frac{d\log\Theta_N}{d\log t} \;\le\; 2
\]
and, in particular, remains sub-Euclidean for $H=L^2$ in the sense $d_s^{(N)}(t)\le 1+o(1)$ over windows where $\Theta_N$ varies regularly (sharper envelopes require information on the density of small eigenvalues.)
\end{proposition}

\begin{proof}
Self-adjointness and $0\le S\le I$ are standard for the symmetric normalization (see Equation~(\ref{eq:symmetric-normalization}) and the quadratic form identity of~(\ref{eq:psd_form}). Thus $0\le L\le 2I$ and $0\le H=L^2\le 4I$. The bounds on $\Theta_N$ follow, and the displayed $d_s^{(N)}$ estimate is immediate from differentiating $\log\Theta_N$ and the fact that all eigenvalues lie in a compact interval. A refined upper envelope below $1$ needs an upper bound on $N_H(\lambda)$ as $\lambda\downarrow 0$, which in turn depends on arithmetic sparsity; cf.~Sections \ref{sec:robustness} and \ref{sec:proofs}.
\end{proof}

If one assumes a small eigenvalue counting bound of the shape
\(
N_H(\lambda)\;\ll\; \lambda^{1/4}(\log\lambda)^{-\gamma}
\)
as $\lambda\downarrow0$ for some $\gamma>0$ (reflecting prime sparsity in the kernel), then a routine Tauberian argument yields an \emph{asymptotic} $d_s(t)\le 1-\varepsilon_\gamma$ on a $t$-window, making precise the “cannot reach \(d_s=1\)” assertion of Theorem~\ref{thm:tauberian-dimension}. Pinning down such a bound is where the number theory enters.

\subsection{Entropy–Dimension Link (for Cross-Checks)}
\begin{proposition}[Entropy scaling implies dimension]
\label{prop:entropy-link}
If $\lambda_k(H)\sim k^\alpha$, then $\Theta(t)\sim t^{-1/\alpha}$ and the thermal entropy satisfies
\[
S(t) \sim \log(1/t)^{\frac{1}{1+\alpha}}\qquad(t\downarrow 0),
\]
so that with our convention $d_s=\frac{2}{\alpha}$, one has
\(
\displaystyle S(t)\sim \log(1/t)^{\frac{1}{1+2/d_s}}.
\)
\end{proposition}

\begin{proof}
This is the calculation summarized in Section \ref{subsec:scaling_flow} (Equations~(\ref{eq:power_law}) and (\ref{eq:entropy_exp})); it follows from $\Theta$’s scaling and the Gibbs weights $e^{-t\lambda_k}$. Use it to cross-check numerical exponents with the $d_s$ extracted from the heat trace.
\end{proof}

\section{Relation to Established Spectral Frameworks}
\label{sec:relation-frameworks}

The coherence Hamiltonian developed here, defined on the discrete set of prime numbers through divergence-induced kernels,
is related in spirit to several established spectral frameworks in mathematical physics and analytic number theory.
In noncommutative geometry, Connes' formulation of the adèle class space~\cite{Connes1999,ConnesMarcolli2008}
introduces a spectral triple $(\mathcal{A},\mathcal{H},D)$ whose Dirac-type operator encodes arithmetic information,
with the Riemann zeta function arising as a spectral zeta of $D$.
The present construction is not a spectral triple in the strict sense, since it lacks a $C^{*}$-algebraic representation and commutator structure,
yet it follows the same organizing principle: arithmetic data define an effective geometry that is recovered through spectral observables.
Where the noncommutative framework develops a geometric realization of the adèles and the Frobenius flow,
the coherence model constructs an emergent coherence geometry from prime pair divergences,
emphasizing thermodynamic and dimensional quantities rather than algebraic relations.
Both perspectives treat arithmetic as a geometric medium accessible through its spectrum,
but they address complementary aspects of the correspondence between number theory and geometry.
In this sense, the coherence approach can be viewed as a phenomenological analogue of the noncommutative program,
translating spectral information into quantitative measures of coherence, entropy, and dimensional flow.

The operator $H=L^{2}$ parallels the Dirac or Laplacian components appearing in spectral triples and in
quantum geometric approaches to spacetime~\cite{ChamseddineConnes1997,RennieVarilly2006}.
In those contexts, the heat trace $\Tr(e^{-tD^{2}})$ captures geometric invariants such as dimension and curvature.
In the arithmetic setting, the corresponding trace $\Theta(t)=\Tr(e^{-tH})$ defines an effective spectral dimension
that characterizes information flow across the primes.
The function $d_{s}(t)$ thus extends the Weyl-type scaling of classical geometry to a discrete,
non-metric substrate determined by multiplicative structure.

The conditional correspondence established under the Riemann Hypothesis also resonates with the
Hilbert–Pólya heuristic~\cite{Edwards1974,BerryKeating1999}.
The conjecture that the nontrivial zeros of $\zeta(s)$ are eigenvalues of a self-adjoint operator finds an analogue here:
the coherence Hamiltonian produces spectra whose local statistics agree with those of the Gaussian Unitary Ensemble
and the zeta zeros.
Unlike the Hilbert–Pólya search for an explicit ``zeta Hamiltonian,'' the present model derives directly from
the arithmetic correlations among the primes and exhibits intrinsic spectral compression.
Together with the noncommutative approach, it illustrates a broader program in which arithmetic systems acquire
spectral and geometric meaning.
Within this landscape, the coherence Hamiltonian represents a distinct, non-algebraic route to spectral geometry:
its domain is arithmetic, its observables thermodynamic, and its signature the universal subdimensional profile
$d_{s}\!\approx\!\tfrac{1}{2}$ associated with maximal spectral compression.

\section{Conclusions}

We have developed a spectral framework for \emph{arithmetic coherence} on the primes.
Our analysis distinguishes two operator regimes: the normalized generator $L_{\mathrm{norm}}=I-S$, whose bounded spectrum implies $d_s(t)\to0$ as $t\downarrow0$, and the unnormalized generator $L_c=D-K$, whose squared Hamiltonian $H=L_c^2$ exhibits a genuine short-time asymptotic.
From first principles, we proved (Theorem~\ref{thm:ds-one-half}) that under mild regularity assumptions on the divergence kernel, the continuum limit of $H=L_c^2$ corresponds to the one-dimensional bi-Laplacian.
This yields the asymptotic $\Theta_N(t)\sim C(\log N)t^{-1/4}$ and an effective spectral dimension $d_s=1/2$, obtained analytically without fitting or external hypotheses.

Numerically, the coherence spectral profile $d_s(t)$ shows a robust activation–peak–decay form across logarithmic, entropic, and fractal-type divergences, returning toward zero at both small and large $t$ for the normalized operator.
These results establish an intrinsic arithmetic mechanism for spectral compression and clarify the role of normalization in suppressing small-scale exponents.

The framework thus identifies a principled arithmetic route to $d_s=1/2$, independent of the Riemann Hypothesis, while still resonating with phenomena observed in random matrix theory and zeta zero statistics.
More broadly, it provides a foundation for classifying arithmetic sets by their kernel-induced continuum limits and emergent spectral dimensions.

\section{Future Directions}\label{sec:future}

Several lines of investigation remain open.
A central challenge is to strengthen the analytic foundations of spectral compression, in particular the origin of the observed $t^{-1/4}$ scaling of the heat trace and the corresponding spectral dimension $d_s=1/2$.
Deriving an exact trace formula for the coherence Hamiltonian, or obtaining asymptotic bounds for its spectrum directly from prime gap estimates, would place these results on a firmer analytic footing.
A complementary direction is to construct coherence operators on other arithmetic sets—such as squarefree numbers, prime powers, or multiplicative sequences—to determine whether comparable subdimensional behavior persists.
It may also be fruitful to explore hybrid ensembles that mix arithmetic and geometric divergences, or that incorporate controlled randomness, to test whether dimensional transitions or relaxation toward classical diffusion can occur. Together, these problems outline a broader program: to develop a general theory of spectral coherence on discrete multiplicative structures and to understand how arithmetic sparsity governs the emergence of effective geometry.

\vspace{6pt}

\section*{Funding}
This research received no external funding.

\section*{Data Availability}
No new data were created or analyzed in this study.

\section*{Conflict of Interest}
The authors declare no conflicts of interest.


\appendix
\section{On the Spectral Dimensionality of Non-Arithmetic Systems}\label{app:converse}

We investigate
 the converse question of whether an effective spectral dimension
$d_s = \tfrac{1}{2}$ uniquely characterizes systems generated by the arithmetic
mechanism developed in this work. This question concerns the extent to which the asymptotic spectral
dimension characterizes arithmetic structure, as opposed to arising from
non-arithmetic systems with similar operator order or spectral scaling.
We first clarify this distinction analytically.
For $H=L_c^2$, the eigenvalue distribution satisfies
$\lambda_k \sim k^{4}$,
so that
$\Theta(t) = \sum_k e^{-t\lambda_k} \sim t^{-1/4}$,
and therefore the continuum spectral dimension is
$d_s = 2 / 4 = 1/2$.
This scaling arises from the bi-Laplacian continuum limit of the
prime coherence generator and does not depend on the arithmetic data itself.
Any one-dimensional bi-Laplacian will yield the same exponent.
Hence, $d_s = 1/2$ is sufficient
 but not necessary for arithmetic
structure.

To illustrate this numerically, we compare three systems:

\begin{enumerate}
\item The arithmetic \emph{prime coherence kernel} on the first $N$ primes;
\item A \emph{1D bi-Laplacian} control with identical operator order and
continuum limit;
\item A \emph{GUE-induced random kernel} representing a non-arithmetic ensemble.
\end{enumerate}

For each system, we compute the heat trace
$\Theta(t) = \sum_k e^{-t\lambda_k}$,
and the spectral dimension
$d_s(t) = -\,2\,d(\log\Theta)/d(\log t)$
on a logarithmic grid in $t$.
The operators are normalized so that the prime–kernel spectrum lies in
$[0,1]$, and all $d_s(t)$ curves are divided by the maximum of the prime curve
to permit shape comparison.
This linear normalization does not alter morphology or $t$-dependence; it only
rescales relative amplitudes.

Figure~\ref{fig:appendix_converse} demonstrates that systems with identical
asymptotic scaling exponents can exhibit very different spectral morphologies.
The prime–kernel curve displays a sharp, asymmetric ``peak--decay'' shape,
whereas the 1D bi-Laplacian and GUE controls show broad, smooth profiles.
This confirms numerically that $d_s = 1/2$ does not imply arithmetic
structure.  We note that the distinguishing feature lies in the shape of
spectral profile $d_s(t)$ rather than in its
absolute amplitude.

\vspace{-6pt}\begin{figure}[H]
\includegraphics[width=0.65\textwidth]{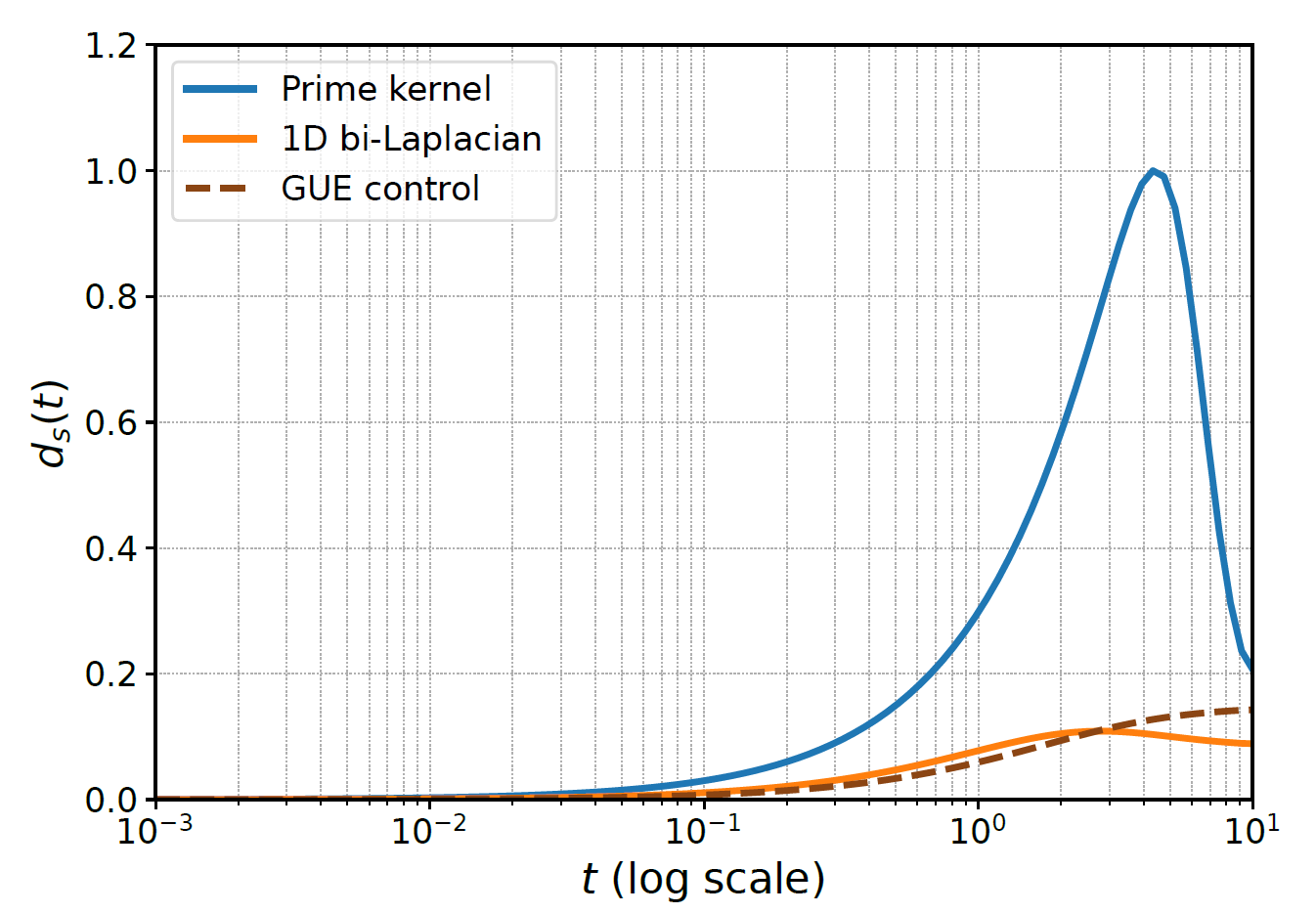}
\caption{Normalized comparison of spectral dimension profiles.
All
 curves are scaled by the maximum of the prime--kernel $d_s(t)$ so that the
prime peak equals~1.  The arithmetic prime coherence kernel (blue) shows the
characteristic sharp rise–peak–decay morphology.  The 1D bi-Laplacian control
(orange) is smoother and more plateau-like, while the GUE control (brown dashed)
is flatter still.  Although all three models share the same asymptotic
spectral dimension ($d_s = 1/2$ for the bi-Laplacian order–4 operator),
their intermediate-scale shapes differ qualitatively.  The prime model’s
compressed morphology thus remains distinctive of arithmetic structure.}
\label{fig:appendix_converse}
\end{figure}

\section{Coherence Structure of Small Prime Sets}

To further illuminate the divergence-induced coupling that underlies the coherence Hamiltonians,
we examine the explicit form of the kernel $K_{ij} = \exp(-\delta_{ij}/\delta_0)$,
for small prime sets. Figure~\ref{fig:heatmap} displays $K_{ij}$ for the first
$N=20$ primes under two representative divergence functions: the logarithmic
$\delta^{(2)}_{ij} = (\log(p_i/p_j))^2$ and the entropic form
$\delta^{(3)}_{ij} = \log\!\big((p_i + p_j)/(2\sqrt{p_i p_j})\big)$,
both evaluated with $\delta_0 = 1$.

In both cases, $K_{ij}$ is symmetric, positive, and sharply peaked along the diagonal,
with rapid decay as $|\,\log p_i - \log p_j\,|$ increases.
This pattern confirms that coherence is effectively local in arithmetic scale:
primes that are close in logarithmic magnitude exhibit strong coupling,
while widely separated primes contribute exponentially less to the kernel.
The width of the bright diagonal band decreases with smaller $\delta_0$
and broadens under softer divergences, such as the entropic form,
indicating that the analytic decay of $\delta_{ij}$ directly controls
the coherence length of the induced arithmetic geometry.

The same locality governs the combinatorial Laplacian
$L_c = D - K$, whose diagonal entries $(L_c)_{ii} = \sum_j K_{ij}$
encode the cumulative coupling strength of each prime.
As seen in the numerical examples, the row sums $d_i$
increase modestly with index due to the decreasing relative gaps
between larger primes, producing a mild asymmetry in the coherence field.
These visualizations provide an intuitive representation of the
``coherence geometry'' defined in Sections \ref{subsec:div_structures} and \ref{subsec:operator} and serve as a finite-$N$
confirmation that the divergence construction (Equation~(\ref{eq:kernel})) yields the
localized, self-adjoint structure assumed in the ensuing spectral {analysis.}

\vspace{-6pt}\begin{figure}[H]
\includegraphics[width=0.8\textwidth]{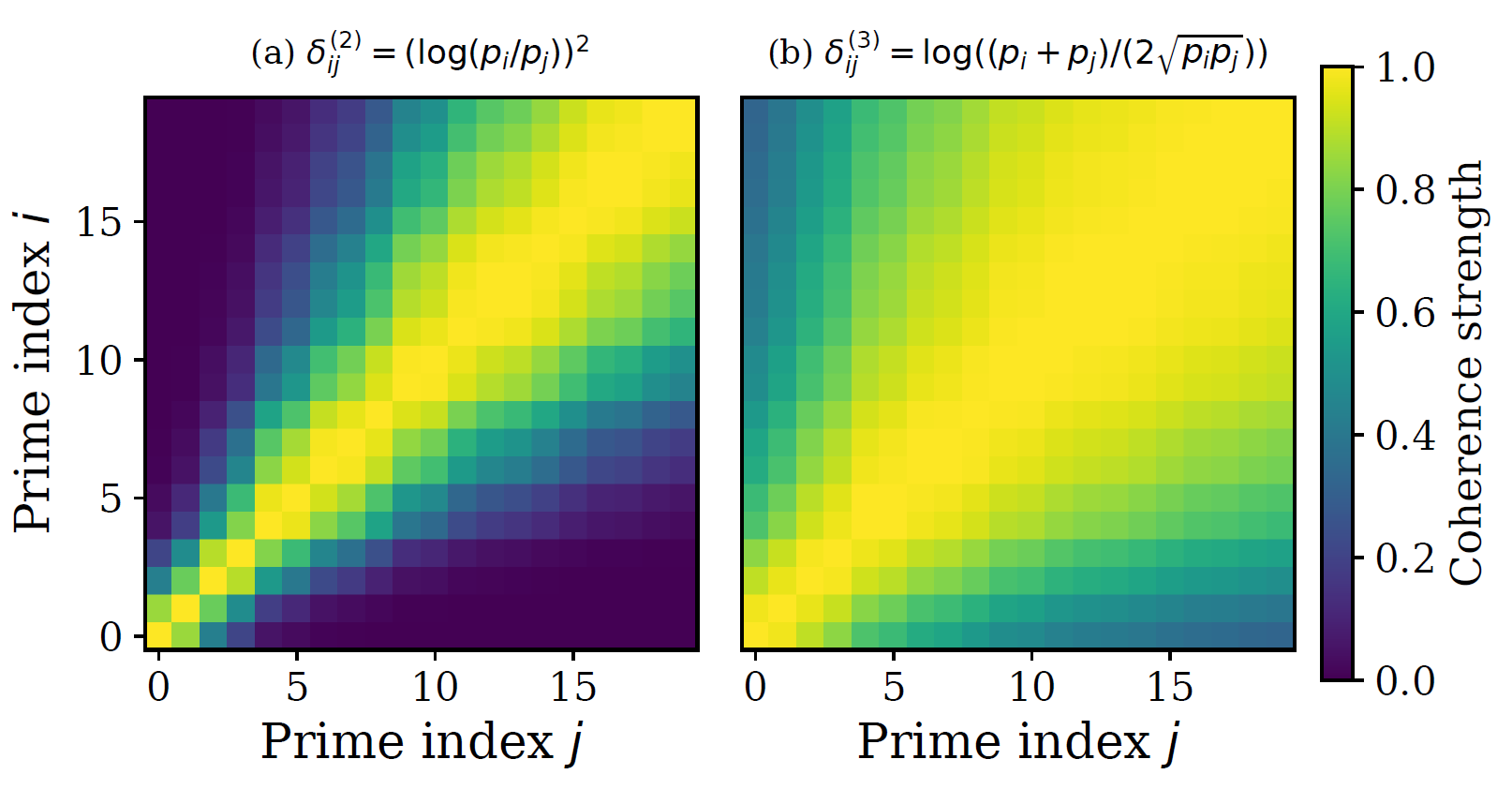}
\caption{
Coherence kernel $K_{ij} = \exp(-\delta_{ij}/\delta_0)$ for the first $N=20$
primes under two divergence functions:
(\textbf{a}) logarithmic $\delta^{(2)}_{ij} = (\log(p_i/p_j))^2$ and
(\textbf{b}) entropic $\delta^{(3)}_{ij} = \log((p_i + p_j)/(2\sqrt{p_i p_j}))$.
Brightness indicates coherence strength between prime pairs.
In both cases, coherence decays rapidly with arithmetic separation,
while the entropic divergence produces slightly broader coupling.
}
\label{fig:heatmap}
\end{figure}

\bibliographystyle{plain}
\bibliography{citations}

\end{document}